\documentclass[10pt]{article}

 \usepackage{amsmath,amsthm,amssymb}
 \usepackage{makeidx,epsfig,lscape}
 \usepackage{color,colortbl}
 \usepackage{fancyhdr}
 \usepackage{xcolor,pict2e}

 \renewcommand{\headrulewidth}{0pt}
 \renewcommand{\footrulewidth}{0.5pt}

 \definecolor{myaqua}{rgb}{0.0,0.5,0.55}
 \definecolor{lightaqua}{rgb}{0.75,0.95,0.95}

 \usepackage[colorlinks = true,
            linkcolor = myaqua,
            urlcolor  = blue,
            citecolor = myaqua]{hyperref}

\usepackage{caption}
\usepackage{floatrow}
 \renewcommand{\figurename}{\color{myaqua}{\bf Figure}}
 \renewcommand{\tablename}{\color{myaqua}{\bf Table}}

\newtheorem{theorem}{Theorem}
\newtheorem{proposition}[theorem]{Proposition}
\newtheorem{lemma}{Lemma}

\def\lin#1#2{\textcolor[rgb]{0.6,0.6,0.6}{\vspace*{#1mm} \hrule
   height 3 pt \vspace*{#2mm}}}
\def\bt{\begin{tabular}}
\def\et{\end{tabular}}
\def\and{\mbox{ and }}
\def\E{\mbox{\bf E}}

\def\1{{\bf 1}}

 \def\sectionn#1{\refstepcounter{section}{\color{myaqua}

 \vskip 6mm

 \noindent\Large\bf\thesection. #1}

 \vskip 3mm}
 \def\subsectionn#1{\refstepcounter{subsection}{\color{myaqua}

 \vskip 5mm

 \noindent\large\bf\thesubsection. #1}

 \vskip 2mm}

\begin{document}

 $\mbox{ }$

 \vskip 12mm

{ 


{\noindent{\huge\bf\color{myaqua}
 Variable selection in multiple regression  with random design
}}
%
\\[6mm]
{\large\bf Alban Mbina Mbina$^1$, Guy Martial Nkiet$^2$, Assi Nguessan$^3$}
\\[2mm]
{ 
 $^1$URMI, Universit\'e des Sciences et Techniques de Masuku,  Franceville, Gabon\\
Email: \href{mailto:albanmbinambina@yahoo.fr}{\color{blue}{\underline{\smash{albanmbinambina@yahoo.fr}}}}\\[1mm]
$^2$URMI, Universit\'e des Sciences et Techniques de Masuku,  Franceville, Gabon\\
Email:
\href{mailto:gnkiet@hotmail.com}{\color{blue}{\underline{\smash{gnkiet@hotmail.com}}}}
 \\[1mm]
$^3$Universit\'e des Sciences et Technologies de Lille,  Lille, France\\
Email:
\href{mailto:assi.nguessan@polytech-lille.fr}{\color{blue}{\underline{\smash{assi.nguessan@polytech-lille.fr}}}}
 \\[4mm]

\lin{5}{7}

 { 
 {\noindent{\large\bf\color{myaqua} Abstract}{\bf \\[3mm]
 \textup{ We propose a method for  variable selection in  multiple regression with random predictors. This method is based on a criterion that  permits to reduce the variable selection problem to a problem of estimating  suitable permutation and dimensionality.  Then, estimators for these parameters are proposed and the resulting method for selecting variables is shown to be consistent. A simulation study that permits to gain understanding of the performances of the proposed approach  and to compare it with an existing method is given. 
  }}}
\\[4mm]
 {\noindent{\large\bf\color{myaqua} Keywords}{\bf \\[3mm]
 Variable selection; Multiple linear regression; Random design; Selection criterion;  Consistency
}

\lin{3}{1}

\sectionn{Introduction}
The selection of variables and models is an old and important problem in statistics, and several approaches have been proposed  to deal with it for various methods of multivariate statistical analysis. For linear regression, many model selection criteria have been proposed in the literature. Surveys on earlier work in this field may be found in \cite{Hock76,Thom78a,Thom78b}, whereas some monographs on this topic are avalaible (e.g.,\cite{ Lin86, Mill90}).  Most of the methods that have been proposed for variable selection in linear regression  deal with the case where the covariates are assumed to be nonrandom; for this case, many selection criteria have been introduced in the literature. These include the FPE criterion (\cite{Thom78a, Thom78b, Shib84, Zhang92}), cross-validation (\cite{Zhang93,Shao93}), AIC and $C_p$ type criteria (e.g., \cite{Fuji97}), the prediction error criterion (\cite{Fuji11}), and so on.  There is just a few works dealing with the case where the covariates are random, although its importance that have been recognized in \cite{Brei92} who argued  that that this case typically gives higher prediction errors than the fixed design counterparts and hence more is gained by variable selection. Linear regression with random design were considered in \cite{Zheng97, Nk01} for variable selection, but these works only deal with univariate models, that is models for which the response is a real random variable.  A recent work that considered multiple regression model is \cite{An13} in which a method  based on applying an  adaptative LASSO type penalty and a novel BIC-type selection criterion have been proposed in order to select both predictors and responses. 

In this paper we extend the approach introduced in \cite{Nk01} to the case of multiple regression. In Section 2, the multiple regression model that is used 
is presented   as well as a statement of the variable selection problem. Then, the used criterion is introduced and we give a characterization result that
permits to reduce the variable selection problem to an estimation problem for two parameters. In Section 3, we propose our method for selecting variables by estimating the two previous  parameters, and we prove its consistency. Section 4 is devoted to a simulation study which permits to evaluate finite sample performances of the proposal and to compare it with the method given in \cite{An13}.  Proofs of lemmas and theorems are given in Section 5.

{ \fontfamily{times}\selectfont
 \noindent

\renewcommand{\headrulewidth}{0.5pt}
\renewcommand{\footrulewidth}{0pt}

 \renewcommand{\headrule}{\hbox to\headwidth{\color{myaqua}\leaders\hrule height \headrulewidth\hfill}}

\sectionn{Model and criterion for selection}
\label{sec2}

{ \fontfamily{times}\selectfont
 \noindent
In this section, the multiple regresion model in which we are interested is introduced and a statement of the corresponding variable selection problem is given. It is described as a problem of estimation of a suitable set. A criterion permitting to characterize this set is propsed as well as an estimator of this criterion. Finally, we give a result that permits to obtain asymtotic properties of this estimator.

\subsectionn{Model and statement of the problem}
\label{subsec:MSP}

{ \fontfamily{times}\selectfont
 \noindent
We consider the multiple regression model given by:
\begin{equation}\label{eq:1}
Y = BX + \varepsilon
\end{equation}
where $X$ and $Y$ are random vectors valued into  $\mathbb{R}^{p}$ and $\mathbb{R}^{q}$ respectively with  $p \geq 2$ and  $q \geq 2$, $B$ is a  $q \times p$ matrix of real  coefficients, and $\varepsilon$ is a random vector valued into $\mathbb{R}^{q}$ and which is independent of $X$. Writing 
\[
X = \left(
 \begin{array}{c}
 X_1\\
 \vdots \\
 X_p
 \end{array}
\right),\,\,\,
Y = \left(
 \begin{array}{c}
 Y_1\\
 \vdots \\
 Y_q
 \end{array}
\right),\,\,\,
\varepsilon = \left(
 \begin{array}{c}
 \varepsilon_1\\
 \vdots \\
 \varepsilon_q
 \end{array}
\right)
\]
and
\[
B = \left(
 \begin{array}{cccc}
  b_{11} &b_{12} &\cdots &b_{1p}\\
  b_{21} &b_{22} &\cdots &b_{2p}\\
  \vdots &\vdots &\cdots &\vdots\\
  b_{q1} &b_{q2} &\cdots &b_{qp}
\end{array}
\right)
\]
it is easily seen that Model  $(\ref{eq:1})$ is equivalent to having a set of $p$ univariate regression models given by:

\begin{equation}\label{eq:2}
Y_i = \sum_{j=1}^{p}b_{ij}X_j + \varepsilon_i,\,\,\,\,\, i=1,\cdots,q,
\end{equation}
and can also be writen as
\begin{equation}\label{eq:3}
Y=\sum_{j=1}^pX_j\textrm{\textbf{b}}_{\bullet j}+\varepsilon
\end{equation}
where
\[
\textrm{\textbf{b}}_{\bullet j} = \left(
 \begin{array}{c}
  b_{1j}\\
  b_{2j}\\
  \vdots \\
  b_{qj}
 \end{array}
\right).
\]
We are interested with the variable selection problem, that is identifying the $X_j$'s which are not relevant in the previous set of models, on the basis of an i.i.d. sample $\left(X^{(k)},Y^{(k)}\right)_{1\leq k\leq n}$ of $(X,Y)$. We say that a variable $X_j$ is not relevant if the corresponding coefficients vector  $\textrm{\textbf{b}}_{\bullet j}$ is null. So, putting  $I = \left\{1, \cdots , p\right\}$ we consider the subset  $I_{0} = \left\{j \in I\,/\, \|\textrm{\textbf{b}}_{\bullet j}\|_{\mathbb{R}^q}= 0  \right\}$ which is assumed to be non-empty, and
we tackle the variable selection problem in Model  $(\ref{eq:1})$ as a problem of estimating  the set $I_0$ or, equivalently, the set $I_1=I-I_0$. In order to simplify the estimation of $I_1$ we will first characterize it  by means of a criterion which introduced below.
}
\subsectionn{Characterization of $I_1$}
\label{subsec:carac}

{ \fontfamily{times}\selectfont
 \noindent
Without loss of generality, we assume that $X$ and $Y$ are centered; thus, that  is also the case for $\varepsilon$. Furthermore, denoting by $\Vert\cdot\Vert_{\mathbb{R}^k}$  the usual Euclidean norm of $\mathbb{R}^k$, we assume that $\mathbb{E}\left(\Vert X\Vert^4_{\mathbb{R}^p}\right)<+\infty$ and  $\mathbb{E}\left(\Vert Y\Vert^4_{\mathbb{R}^q}\right)<+\infty$.  Then, it is possible to define
the  covariance operators
\begin{equation}\label{covop}
V_{1}=\mathbb{E}\left(  X\otimes X\right)  \,\,\,\textrm{ and }\,\,\,V_{12}=\mathbb{E}\left(
Y\otimes X\right),
\end{equation}
where $\otimes$ denotes the tensor product of vectors defined as follows:  when $E$ and $F$ are euclidean spaces and
$\left(  u,v\right)  $ is a pair belonging to $E\times F$, the tensor product
$u\otimes v$ is the linear map from $E$ to $F$ such that
\[
\forall\,h\in E,\,\,\,\left(  \,u\otimes v\right)  \left(  h\right)
=\left\langle u,h\right\rangle _{E}\,\,v,\,\,
\]
where $\left\langle \cdot,\cdot\right\rangle _{E}$ denotes the inner product
in $E$.

\bigskip

\noindent {\bf Remark 1. }  
In all of the paper, we essentially use covariance operators, but the
translation into matrix terms  is obvious  and
more details can be found in  \cite{Daux94}. Particularly, when $u$ and $v$ are vectors in $\mathbb{R}^p$ and $\mathbb{R}^q$ respectively, the matrix related to  the operator
$u\otimes v$, relative to canonical bases, is $vu^T$  where $u^{T}$ is the transpose of
$u$. So, if matricial expressions are prefered to operators, one can identify the operators given in (\ref{covop}) with the  matrices
$V_{1}=\mathbb{E}\left(  XX^{T}\right)$ and $V_{12}=\mathbb{E}
\left(  XY^{T}\right)$.

\bigskip

\noindent In all of the paper,  the operator   $V_1$ is assumed to be  invertible. For any subset $K$ of $I$, let $A_{K}$ be
the  projector 
\[
x=\left(  x_{i}\right)  _{i\in I}\in\mathbb{R}^{p}\mapsto
x_{K}=\left(  x_{i}\right)  _{i\in K}\in\mathbb{R}^{card\left(  K\right)  }
\]
and put $\Pi_{K}:=A_{K}^\ast\left(  A_{K}V_1A_{K}^\ast\right)  ^{-1}A_{K}$, where $A^\ast$ denotes the adjoint operator of $A$. Then, we introduce the criterion
\begin{equation}\label{eq:3}
\xi_{K} = \| V_{12} - V_1 \Pi_{K} V_{12} \|
\end{equation}
where $\Vert\cdot\Vert$ denotes the usual operator  norm given by $\Vert A\Vert=\sqrt{\textrm{tr}\left(A^\ast A\right)}$. This criterion permits to give a more explicit expression of $I_1$ as stated in the following lemma.

\begin{lemma}
We have $I_1\subset K$ if, and only if, $\xi_{K} = 0$.
\end{lemma}
\noindent This lemma permits to characterize the fact that an interger $i$ belongs to $I_0$. Indeed, since having $i\in I_0$ is equivalent to having  $I_1\subset I-\{i\}$, we deduce from this lemma that one has  $i\in I_0$ if, and only if,  $\xi_{K_i}=0$ where $K_i=I-\{i\}$. Then $I_1$ consists of the elements of $I$ for which $\xi_{K_i}$ does not vanish. Now, let us consider the unique permutation $\sigma$ \ of $I$ satisfying:

\bigskip

\textit{(i)} \ $\xi_{K\sigma\left(  1\right)  }\geq\xi_{K\sigma\left(
2\right)  }\geq\cdots\geq\xi_{K\sigma\left(  p\right)  };$

\textit{(ii)} \ $\xi_{K\sigma\left(  i\right)  }=\xi_{K\sigma\left(  j\right)
}$ and $i<j\;$imply $\sigma\left(  i\right)  <\sigma\left(  j\right)  $.

\bigskip

\noindent Since $I_0$ is a not empty, there exists an integer $s\in I$, that we call \textit{the dimensionality}, satisfying 
\[
\xi_{K\sigma\left(  1\right)  }\geq\xi_{K\sigma\left(
2\right)  }\geq\cdots\geq\xi_{K\sigma\left(  s\right)}>0=\xi_{K\sigma\left(  s+1\right)  }=\cdots=\xi_{K\sigma\left(  s+1\right)  }.
\]
Therefore, we obviously have  the following characterization of $I_1$:

\begin{lemma} $I_1=\{\sigma (k)\,/\,1\leq k\leq s\}$.
\end{lemma}

\noindent This result shows that estimation of $I_1$ reduces to that of the two parameters $\sigma $ and $s$. So, our method for selecting variables will be based on estimating these parameters; in the next subsection, an estimator of the used criterion will be introduced. That will be the basis of the proposed procedure for variable selection.
}
\subsectionn{Estimation of the criterion}
\label{subsec:estim}

{ \fontfamily{times}\selectfont
 \noindent
Recalling that we have an  i.i.d. sample $\left(X^{(k)},Y^{(k)}\right)_{1\leq k\leq n}$ of $(X,Y)$, we consider the sample means

\[
\overline{X}^{(n)} = n^{-1}\sum_{k=1}^{n}X^{(k)}, \hspace{0.2cm} \overline{Y}^{(n)} = n^{-1}\sum_{k=1}^{n}Y^{(k)},
\]
and the empirical covariance operators
\[
\widehat{V}_{1}^{(n)} = n^{-1}\sum_{k=1}^{n}(X^{(k)} - \overline{X}^{(n)})\otimes(X^{(k)} - \overline{X}^{(n)}), 
\]
and
\[
\widehat{V}_{12}^{(n)} = n^{-1}\sum_{k=1}^{n}(Y^{(k)} - \overline{Y}^{(n)})\otimes(X^{(k)} - \overline{X}^{(n)}).
\]
Then, for any $K\subset I$, an estimator of $\xi_K$ is given by

\[ 
\widehat{\xi}_{K}^{(n)} = \| \widehat{V}_{12}^{(n)} - \widehat{V}_{1}^{(n)}\widehat{\Pi}^{(n)}_K \widehat{V}_{12}^{(n)}\|
\]
where $\widehat{\Pi}^{(n)}_K = A_{K}^{*}(A_{K}V_1^{(n)} A_{K}^{*})^{-1}A_{K}$. The  result given below  permits to obtain asymptotic properties of this  estimator. As usual, when $E$ and  $F$ are Euclidean vector spaces, we denote by $\mathcal{L}(E,F)$ the vector space of operators from $E$ to $F$. When $E=F$, we simply write $\mathcal{L}(E)$ instead of $\mathcal{L}(E,E)$. Each element $A$ of $\mathcal{L}(\mathbb{R}^{p+q})$ can be writen as
\[
A=\left(
\begin{array}{ccc}
A_{11} &  & A_{12}\\
 & & \\
A_{21} &  & A_{22}
\end{array}
\right)
\]
where $ A_{11}\in\mathcal{L}(\mathbb{R}^p)$,   $ A_{12}\in\mathcal{L}(\mathbb{R}^q,\mathbb{R}^p)$,  $ A_{21}\in\mathcal{L}(\mathbb{R}^p,\mathbb{R}^q)$ and  $ A_{22}\in\mathcal{L}(\mathbb{R}^q)$. Then we consider the projectors
\[
P_1\, :\,
A\in \mathcal{L}(\mathbb{R}^{p+q}) \mapsto A_{11}\in \mathcal{L}(\mathbb{R}^{p})
\,\,\,\textrm{ and }\,\,\,
P_2\, :\,
A\in \mathcal{L}(\mathbb{R}^{p+q}) \mapsto A_{12}\in \mathcal{L}(\mathbb{R}^{q},\mathbb{R}^{p}),
\]
and we have:
\begin{proposition} 
We have 
\[
\sqrt{n}\,\widehat{\xi}_{K}^{(n)}=\|\widehat{\Psi}_{K}^{(n)}(\widehat{H}^{(n)}) + \sqrt{n}\,\delta_{K}\|, 
\]
where 
$\delta_K=V_{12} - V_1 \Pi_{K} V_{12}$, 
$(\widehat{\Psi}_{K}^{(n)})_{n \in \mathbb{N}^{*}}$ is a sequence of random operators  which  converges almost surely, as $n\rightarrow +\infty$,   to  the operator $\Psi_{K}$ of $\mathcal{L}(\mathcal{L}(\mathbb{R}^{p+q}),\mathcal{L}(\mathbb{R}^q,\mathbb{R}^p))$ given by:
\[
\Psi_{K}(A) = P_2(A) - P_1(A)\Pi_{K}V_{12} + V_1\Pi_{K}P_1(A)\Pi_{K}V_{12} - V_1\Pi_{K}P_2(A),
\]
and $(\widehat{H}^{(n)})_{n \in \mathbb{N}^{*}}$ is a sequence of random variables valued into $\mathcal{L}(\mathbb{R}^{p+q})$ which converges  in distribution to random variable $H$ having a normal distributon with mean $0$ and  covariance operator given by: 
\[
\Gamma=\mathbb{E}\left((Z\otimes Z-V)\widetilde{\otimes}(Z\otimes Z-V)\right),
\] 
$Z$  being  the $\mathbb{R}^{p+q}$-valued random variable given by
\[
Z =\left(\begin{array}{c} 
X\\
Y
\end{array}\right)
\] 
and $\tilde{\otimes}$ is the tensor product between  elements of $\mathcal{L}(\mathbb{R}^{p+q})$ related to the inner product $<A,B>=tr\left(A^\ast B\right)$.
\end{proposition}

}

\sectionn{Selection of variables}
\label{sec:MLE}

{ \fontfamily{times}\selectfont
 \noindent
Lemma 2 shows that estimation of  $I_{1}$
reduces to that of $\sigma$ and $s$. In this section, estimators for these two parameters are
proposed and  consistency properties are established for them.
\subsectionn{Estimation of $\sigma$ and $s$}
\label{subsec:estimpar}

{ \fontfamily{times}\selectfont
 \noindent
Let us consider a sequence $\left(  f_{n}\right)  _{n\in\mathbb{N}^{\ast}}$ of
\ functions from $I$ \ to $\mathbb{R}_{+}$ such that there exists a real
$\alpha\in\left]  0,1/2\right[  $ and a strictly decreasing function
\ $f\;:\;I\rightarrow\mathbb{R}_{+}$ satisfying:
\[
\forall i\in I,\;\;\lim_{n\rightarrow+\infty}\left(  n^{\alpha}\;f_{n}\left(
i\right)  \right)  =f\left(  i\right)  .
\]
Then, recalling that $K_{i}=I-\left\{  i\right\}  $, we put
\[
\widehat{\phi}_{i}^{\left(  n\right)  }=\widehat{\xi}_{K_{i}}^{\left(
n\right)  }+f_{n}\left(  i\right)  \text{\ \ \ \ (}i\in I\text{)}
\]
and we take as estimator of $\sigma$ the random permutation $\widehat{\sigma
}^{\left(  n\right)  }$ of $I$ such that

\[
\widehat{\phi}_{\widehat{\sigma}^{\left(  n\right)  }\left(  1\right)
}^{\left(  n\right)  }\geq\widehat{\phi}_{\widehat{\sigma}^{\left(  n\right)
}\left(  2\right)  }^{\left(  n\right)  }\geq\cdots\geq\widehat{\phi
}_{\widehat{\sigma}^{\left(  n\right)  }\left(  p\right)  }^{\left(  n\right)
}
\]
and if  $\widehat{\phi}_{\widehat{\sigma}^{\left(  n\right)  }\left(
i\right)  }^{\left(  n\right)  }=\widehat{\phi}_{\widehat{\sigma}^{\left(
n\right)  }\left(  j\right)  }^{\left(  n\right)  }$  with  $i<j$,
then $\widehat{\sigma}^{\left(  n\right)  }\left(  i\right)
<\widehat{\sigma}^{\left(  n\right)  }\left(  j\right)$.
Furthermore, we consider the random set  $\widehat{J}_{i}^{\left(  n\right)  }=\left\{
\widehat{\sigma}^{\left(  n\right)  }\left(  j\right)  ;\;1\leq j\leq
i\right\}  $ and the random variable
\[
\widehat{\psi}_{i}^{\left(  n\right)  }=\widehat{\xi}_{\widehat{J}%
_{i}^{\left(  n\right)  }}^{\left(  n\right)  }+g_{n}\left(  \widehat{\sigma
}^{\left(  n\right)  }\left(  i\right)  \right)  \;\;\;\;\;\text{(}i\in
I\text{)}
\]
where $\left(  g_{n}\right)  _{n\in\mathbb{N}^{\ast}}$ is a sequence of
\ functions from $I$ to \ $\mathbb{R}_{+}$ such that there exist a real $\beta
\in\left]  0,1\right[  $ and a strictly increasing function
\ $g\;:\;I\rightarrow\mathbb{R}_{+}$ satisfying:
\[
\forall i\in I,\;\;\lim_{n\rightarrow+\infty}\left(  n^{\beta}\;g_{n}\left(
i\right)  \right)  =g\left(  i\right)  .
\]
Then, we take as estimator of $s$ the random variable
\[
\widehat{s}^{\left(  n\right)  }=\min\left\{  i\in I\;/\;\widehat{\psi}%
_{i}^{\left(  n\right)  }=\min_{j\in I}\left(  \widehat{\psi}_{j}^{\left(
n\right)  }\right)  \right\}  .
\]
 The variable
selection is achieved by taking the random set
\[
\widehat{I}_{1}^{\left(  n\right)  }=\left\{  \widehat{\sigma}^{\left(
n\right)  }\left(  i\right)  \;;\;1\leq i\leq\widehat{s}^{\left(  n\right)
}\right\}
\]
  as estimator of $I_{1}$.

}
\subsectionn{Consistency}
\label{subsec:cons}

{ \fontfamily{times}\selectfont
 \noindent
The following theorem establishes consistency for the preceding estimators :

\begin{theorem}
We have:

(i) $\lim_{n\rightarrow+\infty}P\left(  \widehat{\sigma}^{\left(
n\right)  }=\sigma\right)  =1;$

(ii) $\widehat{s}^{\left(  n\right)  }$  converges in
probability to $s$,  as $n\rightarrow+\infty$.
\end{theorem}

\noindent As a consequence of this theorem, we easily obtain:
$\lim_{n\rightarrow+\infty
}P\left(  \widehat{I}_{1}^{\left(  n\right)  }=I_{1}\right)  =1$. This shows the consistency of our method for selecting variables in the model (\ref{eq:1}).

}

 }

\sectionn{Simulations}
\label{sec:K-M}

{ \fontfamily{times}\selectfont
 \noindent
In this section, we report  results of a simulation study which was made in order to  check the efficacy of the proposed approach and to compare it with an existing method: the ASCCA method introduced by An \textit{et al.} (2013). This latter method is based on re-casting the multivariate regression problem as a classical CCA problem for which a least quares type formulation is constructed, and applying an  adaptative LASSO type penalty together with a BIC-type selection criterion  (see \cite{An13}  for more details).  Our simulated data is based on two independent data sets: training data and test data, each with sample size $n=50,\,100,\,500,\,800,\,1000,\,2000$.  The training data is used for selecting variables by using both  our method, with penalty terms  $f_{n}\left(  i\right)
=n^{-1/4}\;i^{-1}$and \ $g_{n}\left(  i\right)  =n^{-3/4}\;i$,  and the ASCCA method. The test data is used for computing prediction error given by
\[
e = \frac{1}{n}\sum_{k=1}^{n} \| Y^{(k)} - \widehat{Y}^{(k)} \|^{2}_{\mathbb{R}^q},
\]
where $ Y^{(k)}$ is an observed response and $\widehat{Y}^{(k)}$ is the usual linear predictor of $Y^{(k)}$ computed by using the variables selected at the previous step, that is $\widehat{Y}^{(k)}=\left(\mathbb{X}^T\mathbb{X}\right)^{-1}Y^{(k)}$ where $\mathbb{X}$ is a matrix with $n$ rows and columns containing the observations of the $X_j$'s that have been  selected in the previous step. Each data set was generated as follows: $X^{(k)}$ is generated from a multivariate normal distribution in $\mathbb{R}^7$ with mean $0$ and covariance $cov(X^{(k)}_i,X^{(k)}_j)=0.5^{\vert i-j\vert}$ for any $1\leq i,j\leq 7$, and the corresponding response  $ Y^{(k)}$ is generated according to (\ref{eq:1}) with
\[
B=\left(
\begin{array}{ccccccc}
3 & 0 & 0 & 1.5 & 0 & 0 & 2\\
4 & 0 & 0 & 2.5 & 0 & 0 & -1\\
5 & 0 & 0 & 0.5 & 0 & 0 & 3\\
6 & 0 & 0 & 3 & 0 & 0 & 1\\
7 & 0 & 0 &  6 & 0 & 0 & 4
\end{array}
\right)
\]
and the related error term  $\varepsilon^{(k)}$ having a multivariate normal distribution in $\mathbb{R}^5$ with mean $0$ and covariance matrix $0.5\,I_5$, where $I_5$ denotes the $5$-dimensional identity matrix. The outputs of the numerical experiment are the averages of the aforementioned prediction errors over 2000 independent replications.  The results  are reported in Table 1. Our method  gives the better results for  $n\geq 100$ but was outperformed by the ASCCA method for $n=50$.

 \fancyfoot{}
 \fancyfoot[C]{\leavevmode
 \put(0,0){\color{lightaqua}\circle*{34}}
 \put(0,0){\color{myaqua}\circle{34}}
 \put(-5,-3){\color{myaqua}\thepage}}

\begin{table}
\centering \caption{Average of prediction errors over 2000 replications}
 {\begin{tabular}{ccccc} 
\hline\hline
 \rowcolor{lightaqua}  Sample size  &  & Proposed  method &    & ASCCA\\
\hline 
 & & & & \\
 50  &     & 0.00105  &  & 5.323e-6 \\
  100  &  & 0.00012  &  & 0.00052 \\
   500  &  & 1.009e-6 &  & 9.075e-6\\
800  &  & 20602e-7  & & 5.789e-7 \\
    1000  &  & 1.243e-7 &  & 1.308e-7 \\
    2000  &  & 1.436e-8 &  & 1.692e-8 \\[2mm]

\hline
\end{tabular}}
\label{table:simulation_est}
\end{table}

}

\sectionn{Proofs}
\label{subsec:proofs}

{ \fontfamily{times}\selectfont
 \noindent

\subsectionn{Proof of Lemma 1}
\label{subsec:lem1}

{ \fontfamily{times}\selectfont
 \noindent
Denoting by $(\Omega, \mathcal{A},P)$ the considered probability space,  we consider the operators:
\[
L_1 : x =\left(
\begin{array}{c}
x_1\\
\vdots\\
x_p
\end{array}
\right)
 \in \mathbb{R}^{p} \longmapsto \sum_{j=1}^{p}x_jX_j \in L^{2}(\Omega, \mathcal{A},P)\textrm{ and }
L_2: y = \left(
\begin{array}{c}
y_1\\
\vdots\\
y_q
\end{array}
\right) \in \mathbb{R}^{q}  \longmapsto \sum_{i=1}^{q}y_iY_i \in L^{2}(\Omega, \mathcal{A},P)
\]
with  adjoints are  respectively given  by:
\[
L_1^{*} : Z \in L^{2}(\Omega, \mathcal{A},P) \longmapsto \E(ZX) \in \mathbb{R}^{p},
\textrm{ and }L_2^{*} : Z \in L^{2}(\Omega, \mathcal{A},P) \longmapsto \E(ZY) \in \mathbb{R}^{q}.
\]
It is easy to verify that $L_1^{*}L_1 = V_1$ and $L_1^{*}L_2 = V_{12}$. Denoting by $R(A)$ the range of the operator $A$, and from the fact that the orthogonal projector $\Pi_{R(A)}$ onto $R(A)$ is given by $\Pi_{R(A)}=A(A^\ast A)^{-1}A^\ast$,  we clearly have
\begin{eqnarray}\label{xiK}
\xi_{K} = \| L_1^{*}L_2 - L_1^{*}L_1  A_{K}^{*}(A_{K}L_1^{*}L_1 A_{K}^{*})^{-1}A_{K} L_1^{*}L_2 \|
= \|L_1^{*}L_2 -  L_1^{*} \Pi_{R(L_1 A_{K}^{*})} L_2 \| 
        = \| L_1^{*} \Pi_{R(L_1 A_{K}^{*})^{\bot}} L_2 \|,
\end{eqnarray}
where $E^\bot $ denotes the orthogonal space of the vector space $E$. For any vector $\alpha=(\alpha_1,\cdots,\alpha_q)^T$ in $\mathbb{R}^{q}$, one has
\[
L_2(\alpha)=\sum_{i=1}^{q}\alpha_iY_i =\sum_{i=1}^{q}\alpha_i\left(\sum_{j=1}^{p}b_{ij}X_j + \varepsilon_i      \right) = \sum_{i=1}^{q}\sum_{j=1}^{p}\alpha_ib_{ij}X_j + \sum_{i=1}^{q}\alpha_i\varepsilon_i.
\]
Since for any $u=(u_1,\cdots,u_p)^T\in\mathbb{R}^p$, we have 
\[
<L_1(u),\alpha_i\varepsilon_i>=\sum_{j=1}^pu_j<X_j,\alpha_i\varepsilon_i>=\sum_{j=1}^pu_j\alpha_i\mathbb{E}\left(X_j\varepsilon_i\right)=\sum_{j=1}^pu_j\alpha_i\mathbb{E}\left(X_j\right)\mathbb{E}\left(\varepsilon_i\right)=0,
\]
it follows that $\alpha_i\varepsilon_i$ $\in$ $R(L_1)^{\bot}$ and, from $R(L_1)^\bot\subset R(L_1A_K^\ast)^\bot$, we obtain
\[
L_1^{*} \Pi_{R(L_1 A_{K}^{*}))^{\bot}}\alpha_{i} \varepsilon_{i} = L_1^{*} \alpha_{i} \varepsilon_{i}=\mathbb{E}\left(\alpha_i\varepsilon_iX\right)=\alpha_i \E(\varepsilon_i)\E(X) = 0.
\]
Thus,
\begin{eqnarray}\label{egalite}
L_1^{*} \Pi_{R(L_1 A_{K}^{*}))^{\bot}} L_2(\alpha) =
L_1^{*} \Pi_{R(L_1 A_{K}^{*}))^{\bot}}\sum_{i=1}^{q}\sum_{j=1}^{p}\alpha_ib_{ij}X_j
=
\sum_{i=1}^{q}\alpha_iL_1^{*} \Pi_{R(L_1 A_{K}^{*})^{\bot}}L_1(\textrm{\textbf{b}}_{i\bullet}),
\end{eqnarray}
where
\[
\textrm{\textbf{b}}_{i\bullet}=
 \left(
 \begin{array}{c}
  b_{i1}\\
  b_{i2}\\
  \vdots \\
  b_{ip}
 \end{array}
\right).
\]
If $\xi_K=0$, then  considering, for $i=1,\cdots,q$, the vector $\alpha=\left(0,\cdots,0,1,0,\cdots,0\right)$  of $\mathbb{R}^q$ whose coordinates are  null except the $i$-th one which equals $1$, we deduce from (\ref{egalite}) that $L_1^{*} \Pi_{R(L_1 A_{K}^{*})^{\bot}}L_1(\textrm{\textbf{b}}_{i\bullet})=0$. Since, for any operator $A$, ker$(A^\ast A)=$ker$(A)$, it follows that we have $\Pi_{R(L_1 A_{K}^{*})^{\bot}}L_1(\textrm{\textbf{b}}_{i\bullet})=0$, that is 
\begin{eqnarray}\label{inclus}
L_1(\textrm{\textbf{b}}_{i\bullet}) \in R(L_1 A_{K}^{*}).
\end{eqnarray}
Denoting by $\vert K\vert$ the cardinality of $K$ and putting $K=\{k_1,k_2,\cdots,k_{\vert K\vert}\}$, we deduce from (\ref{inclus}) that there exists a vector $\beta=\left(\beta_1,\cdots,\beta_{\vert K\vert}\right)^T\in\mathbb{R}^{\vert K\vert}$ such that $L_1(\textrm{\textbf{b}}_{i\bullet}) =L_1 A_{K}^{*}\beta$, that is
\[
\sum_{j=1}^{p}b_{ij}X_j  = \sum_{\ell=1}^{|K|} \beta_\ell X_{k_\ell}   
\]
and, equivalently,
\begin{eqnarray}\label{eqlineaire}
\sum_{\ell=1}^{\vert K\vert}\left(b_{i k_\ell} - \beta_\ell\right)X_{k_\ell}+\sum_{\ell\in I-K}b_{ij}X_j=0.
\end{eqnarray}
Since $V_1$ is invertible we have $\textrm{ker}(L_1)=\textrm{ker}(L_1^\ast L_1)=\textrm{ker}(V_1)=\{0\}$. Then, $X_1,\cdots,X_p$ are linearly independent and, therefore, (\ref{eqlineaire}) implies that, for all $j\in I-K$, $bij=0$. This property holds for any $i\in\{1,\cdots,q\}$, then we deduce that $I-K\subset I_0$ and, equivalently, that $I_1\subset K$. Reciprocally, we first have 
\begin{eqnarray*}
L_1^{*} \Pi_{R(L_1 A_{K}^{*})^{\bot}}L_1(\textrm{\textbf{b}}_{i\bullet})&=&L_1^{*} \Pi_{R(L_1 A_{K}^{*})^{\bot}}\sum_{j=1}^{p}b_{ij}X_j\\
&=&L_1^{*} \Pi_{R(L_1 A_{K}^{*})^{\bot}}\left(\sum_{j\in K}b_{ij}X_{j}+\sum_{j\in I-K}b_{ij}X_j\right)\\
&=&L_1^{*} \Pi_{R(L_1 A_{K}^{*})^{\bot}}\left(\sum_{\ell=1}^{\vert K\vert}b_{ik_\ell}X_{k_\ell}+\sum_{j\in I-K}b_{ij}X_j\right).
\end{eqnarray*}
If  $I_1\subset K$, then $I-K\subset I_0$ and, consequently, for all $j\in I-K$,  $b_{ij}=0$. Thus
\[
L_1^{*} \Pi_{R(L_1 A_{K}^{*})^{\bot}}L_1(\textrm{\textbf{b}}_{i\bullet})=L_1^{*} \Pi_{R(L_1 A_{K}^{*})^{\bot}}\left(\sum_{\ell=1}^{\vert K\vert}b_{ik_\ell}X_{k_\ell}\right)
=L_1^{*} \Pi_{R(L_1 A_{K}^{*})^{\bot}}L_1A_K^\ast(\textrm{\textbf{b}}_{i\bullet})=0
\]
because $L_1A_K^\ast(\textrm{\textbf{b}}_{i\bullet})\in R(L_1A_K^\ast)$. Then, from (\ref{egalite}) and (\ref{xiK}), we deduce that $\xi_K=0$.

}

\subsectionn{Proof of Proposition 1}
\label{subsec:lem1}

{ \fontfamily{times}\selectfont
 \noindent
We have:
\begin{eqnarray*}
\sqrt{n}\widehat{\xi}_{K}^{(n)} = \|\sqrt{n}(\widehat{V}^{(n)}_{12} - V_{12}) &- &\sqrt{n}(\widehat{V}^{(n)}_{1} - V_1)\widehat{\Pi}^{(n)}_{K}\widehat{V}^{(n)}_{12} -  V_{1}\left(\sqrt{n}(\widehat{\Pi}^{(n)}_{K} - \Pi_{K})\right)\widehat{V}^{(n)}_{12} \\
&-& V_{1}\Pi_{K}\left(\sqrt{n}(\widehat{V}^{(n)}_{12} - V_{12})\right) + \sqrt{n}\delta_{K} \|,
\end{eqnarray*}
and since
\begin{eqnarray*}
\widehat{\Pi}^{(n)}_{K} - \Pi_{K} &=&  A^{*}_{K}\left((A_{K}\widehat{V}^{(n)}_{1}A^{*}_{K})^{-1} - (A_{K}V_{1}A^{*}_{K})^{-1}  \right)A_{K} \\
&=& A^{*}_{K}\left(-(A_{K}\widehat{V}^{(n)}_{1}A^{*}_{K})^{-1} \left(A_{K}\widehat{V}^{(n)}_{1}A^{*}_{K} - A_{K}V_{1}A^{*}_{K}\right) (A_{K}V_{1}A^{*}_{K})^{-1} \right)A_{K} \\
&=& -\widehat{\Pi}^{(n)}_{K}\left(\widehat{V}^{(n)}_{1} - V_{1}\right)\Pi_{K},
\end{eqnarray*}
it follows:
\begin{eqnarray}\label{decomp}
\sqrt{n}\widehat{\xi}_{K}^{(n)} &=& \|\sqrt{n}(\widehat{V}^{(n)}_{12} - V_{12}) - \sqrt{n}(\widehat{V}^{(n)}_{1} - V_1)\widehat{\Pi}^{(n)}_{K}\widehat{V}^{(n)}_{12} \\\nonumber
&+& V_{1}\widehat{\Pi}^{(n)}_{K}\left(\sqrt{n}\left(\widehat{V}^{(n)}_{1} - V_{1}\right)\right)\Pi_{K}\widehat{V}^{(n)}_{12} \\ \nonumber
&-& V_{1}\Pi_{K}\left(\sqrt{n}(\widehat{V}^{(n)}_{12} - V_{12})\right) + \sqrt{n}\delta_{K} \|.
\end{eqnarray}
Let us consider the $\mathbb{R}^{p+q}$-valued  random vectors  
\[
Z = \left(
 \begin{array}{c}
 X\\
 Y
 \end{array}
\right),\,\,\,
Z^{(k)} =\left(
 \begin{array}{c}
 X^{(k)}\\
 Y^{(k)}
 \end{array}\right),\,\,\,k=1,\cdots,n;
\]
the  covariance operator of $Z$ is given by  $V = \mathbb{E}(Z \otimes Z)$ and can be writen as
\begin{eqnarray}\label{covz}
V=\left(
\begin{array}{ccc}
V_1 & & V_{12}\\
   & & \\
V_{21} & & V_2
\end{array}
\right)
\end{eqnarray}
where $V_2= \mathbb{E}(Y \otimes Y)$ and $V_{21}=V_{12}^\ast$. Further, putting
\[
\overline{Z}^{(n)} = n^{-1}\sum_{k=1}^{n}Z^{(k)},\,\,\,\textrm{ and }\,\,\, 
\widehat{V}^{(n)} = n^{-1}\sum_{k=1}^{n}(Z^{(k)} - \overline{Z}^{(n)})\otimes(Z^{(k)} - \overline{Z}^{(n)}), 
\]
we can write
\begin{eqnarray}\label{covempz}
\widehat{V}^{(n)}=\left(
\begin{array}{ccc}
\widehat{V}^{(n)}_1 & & \widehat{V}^{(n)}_{12}\\
   & & \\
\widehat{V}^{(n)}_{21} & &\widehat{V}^{(n)}_2
\end{array}
\right)
\end{eqnarray}
where
$\widehat{V}_{2}^{(n)} = n^{-1}\sum_{k=1}^{n}(Y^{(k)} - \overline{Y}^{(n)})\otimes(Y^{(k)} - \overline{Y}^{(n)})$
and  $\widehat{V}_{21}^{(n)}=\left(\widehat{V}_{12}^{(n)}\right)^\ast$. Then we deduce from (\ref{decomp}), (\ref{covz}) and (\ref{covempz}) that 
$
\sqrt{n}\widehat{\xi}_{K}^{(n)}= \|\widehat{\Psi}_{K}^{(n)}(\widehat{H}^{(n)}) + \sqrt{n}\delta_{K}\|$, 
where $\widehat{H}^{(n)} = \sqrt{n}\left(\widehat{V}^{(n)} - V\right)$ and $\widehat{\Psi}_{K}^{(n)}$ is the random operator from $ \mathcal{L}(\mathbb{R}^{p+q})$ to $ \mathcal{L}(\mathbb{R}^{p})$ defined by 
\[
\forall A\in  \mathcal{L}(\mathbb{R}^{p+q}),\,\,\,\widehat{\Psi}_{K}^{(n)}(A)=P_2(A) - P_1(A)\widehat{\Pi}_{K}^{(n)}\widehat{V}_{12}^{(n)} + V_1\widehat{\Pi}^{(n)}_{K}P_{1}(A)\Pi_{A}\widehat{V}_{12}^{(n)} - V_1\Pi_{K}P_2(A).
\]
Considering the usual operators norm $\Vert\cdot\Vert_\infty$ defined in $\mathcal{L}(E,F)$ by $\Vert A\Vert_\infty=\sup_{x\in E-\{0\}}\Vert Ax\Vert_F/\Vert x \Vert_E$ and recalling that, for two operators $A$ and $B$, one has $\Vert AB\Vert_\infty\leq \Vert A\Vert_\infty\Vert B\Vert_\infty$, we obtain 
\begin{eqnarray*}
\|\widehat{\Psi}_{K}^{(n)}(A) - \Psi_{K}(A) \|_\infty &=& 
\left\Vert -P_1(A)\left(\widehat{\Pi}_{K}^{(n)}-\Pi_K\right)\widehat{V}_{12}^{(n)} - P_1(A)\Pi_K\left(\widehat{V}_{12}^{(n)}-V_{12}\right)\right.\\
& &\left. + V_1\left(\widehat{\Pi}_{K}^{(n)}-\Pi_K\right)P_1(A)\Pi_{K}\widehat{V}_{12}^{(n)}+-  V_1\Pi_{K}P_1(A)\Pi_{K}\left(\widehat{V}_{12}^{(n)}-V_{12}\right)\right\Vert _\infty\\
 &\leq& \|P_1(A)\|_\infty\left[\|\widehat{\Pi}_{K}^{(n)}-\Pi_{K}\|_\infty\|\widehat{V}_{12}^{(n)}\|_\infty
  +\|\Pi_{K}\|_\infty\|\widehat{V}_{12}^{(n)}-V_{12}\|_\infty\right. \\
  & & +\|V_1\|_\infty\|\Pi_{K}\|_\infty\|\widehat{\Pi}_{K}^{(n)}-\Pi_{K}\|_\infty\|\widehat{V}_{12}^{(n)}\|_\infty \\
& & \left. + \|V_1\|_\infty\|\Pi_{K}\|_\infty^{2}\|\widehat{V}_{12}^{(n)}-V_{12}\|_\infty\right] \\
&\leq & \left[\|\widehat{\Pi}_{K}^{(n)}-\Pi_{K}\|_\infty\|\widehat{V}_{12}^{(n)}\|_\infty 
  +\|\Pi_{K}\|_\infty\|\widehat{V}_{12}^{(n)}-V_{12}\|_\infty\right. \\
  & & +\|V_1\|_\infty\|\Pi_{K}\|_\infty\|\widehat{\Pi}_{K}^{(n)}-\Pi_{K}\|_\infty\|\widehat{V}_{12}^{(n)}\|_\infty \\
& & \left. + \|V_1\|_\infty\|\Pi_{K}\|_\infty^{2}\|\widehat{V}_{12}^{(n)}-V_{12}\|_\infty\right]\Vert P_1\Vert_{\infty , \infty} \|A\|_\infty,
\end{eqnarray*}
where $\Vert T\Vert_{\infty,\infty}:=\sup_{A\in\mathcal{L}(\mathbb{R}^{p+q})-\{0\}}\Vert T(A)\Vert_\infty/\Vert A \Vert_\infty$.  Hence
\begin{eqnarray}\label{ineg}
\|\widehat{\Psi}_{K}^{(n)} - \Psi_{K} \|_{\infty ,\infty} &\leq & \left[\|1+\|V_1\|_\infty\|\Pi_{K}\|_\infty\|\right]\|\widehat{V}_{12}^{(n)}\|_\infty \Vert \widehat{\Pi}_{K}^{(n)}-\Pi_{K}\|_\infty\Vert P_1\Vert_{\infty , \infty}\\
& &  +\left[1+ \|V_1\|_\infty\|\Pi_{K}\|_\infty\right]\|\Pi_{K}\|_\infty\|\widehat{V}_{12}^{(n)}-V_{12}\|_\infty\Vert P_1\Vert_{\infty , \infty}  .\nonumber
\end{eqnarray}
From the strong law of large numbers it is easily seen that $\widehat{V}_1^{(n)}$ (resp.   $\widehat{V}_{12}^{(n)}$ converges almost surely, as $n\rightarrow +\infty$ to $V_1$ (resp. $V_{12}$). Therefore, $\widehat{\Pi}_{K}^{(n)}$ converges almost surely, as $n\rightarrow +\infty$ to $ \Pi_{K} $, and from (\ref{ineg}) we deduce that $ \widehat{\Psi}_{K}^{(n)}$  converges almost surely, as $n\rightarrow +\infty$ to $\Psi_{K}$. It remains to obtain the asymptotic distribution of $\widehat{H}^{(n)}$. We  have $ \widehat{H}^{(n)}=\widehat{H}^{(n)}_{1} - \widehat{H}^{(n)}_{2}$ where
\[
 \widehat{H}^{(n)}_{1} = \sqrt{n} \left(\frac{1}{n} \sum_{k=1}^{n}Z_k \otimes Z_k -V \right) \,\,\,\textrm{ and }\,\,\,
 \widehat{H}^{(n)}_{2} = \frac{1}{\sqrt{n}} \left((\sqrt{n}\,\overline{Z}^{(n)}) \otimes (\sqrt{n}\,\overline{Z}^{(n)}) \right).
\]
The central limit theorem ensures that $\widehat{H}^{(n)}_{1}$ (resp. $\sqrt{n}\,\overline{Z}^{(n)}$) converges in distribution, as $n\rightarrow +\infty$, to a random variable $H$ (resp. $U$) having a centered normal distribution with  covariance operator $\Gamma$ (resp. $\Gamma^\prime$)  given by 
\[
\Gamma=\mathbb{E}\left((Z\otimes Z-V)\widetilde{\otimes}(Z\otimes Z-V)\right)\,\,\,\,\textrm{(resp. } \Gamma^\prime= \mathbb{E}\left(Z\otimes Z\right)\textrm{ )}.
\]
Hence, $ \widehat{H}^{(n)}_{2}$ converges in probability, as $n\rightarrow+\infty$, to $0$ and Slustky theorem permits to conclude that  $\widehat{H}^{(n)}$ converges in distribution, as $n\rightarrow +\infty$, to  $H$.

}

\subsectionn{Proof of Theorem 2}
\label{subsec:lem1}

{ \fontfamily{times}\selectfont
 \noindent
We just need to prove the lemma which is given below. Then the proof of Theorem 1 is similar than that of Theorem 3.1 in \cite{Nk12}.  Let $r$ $\in$ $\mathbb{N}^{*}$ and $(m_1,\cdots,m_r)$ $\in$ $(\mathbb{N}^{*})^{r}$ such that  $\sum_{\ell=1}^rm_\ell=p$ and 
\[
\xi_{K_{\sigma(1)}} = \cdots =\xi_{K_{\sigma(m_1)}} > \xi_{K_{\sigma(m_1 +1)}} = \cdots = \xi_{K_{\sigma(m_1+m_2)}} >  \cdots >  \xi_{K_{\sigma(m_1+m_2+\cdots+m_{r-1}+1)}}  = \cdots = \xi_{K_{\sigma(m_1+m_2+\cdots+m_r)}}.
\]
Then, putting  $E=\{\ell\in\mathbb{N}^\ast \,/\,1\leq \ell\leq r,\,\,m_\ell\geq 2\}$ and $F_\ell:=  \left\{\left(\sum_{k=0}^{\ell -1}m_k\right)+1,\cdots,\left(\sum_{k=0}^{\ell }m_k\right)-1  \right\}$ with $m_0=0$, we have:
\begin{lemma}
If $E\neq\emptyset$, then for all $\ell\in E$ and all $i\in F_\ell$, the sequence $n^\alpha\left(\widehat{\xi}^{(n)}_{K_{\sigma (i)}}-\widehat{\xi}^{(n)}_{K_{\sigma (i+1)}}\right)$ converges in probability to $0$ as $n\rightarrow +\infty$.
\end{lemma}
\begin{proof}
Let us put $\gamma_\ell=\xi_{K_{\sigma (i)}}=\xi_{K_{\sigma (i+1)}}$; if $\gamma_\ell=0$, then
\begin{eqnarray*}
\left\vert n^{\alpha}\left( \widehat{\xi}^{(n)}_{K_{\sigma(i)}} -\widehat{\xi}^{(n)}_{K_{\sigma(i+1)}}\right) \right\vert&=& n^{\alpha-\frac{1}{2}}\left\vert\|\widehat{\Psi}^{(n)}_{K_{\sigma(i)}}(\widehat{H}^{(n)}) \|- \| \widehat{\Psi}^{(n)}_{K_{\sigma(i+1)}}(\widehat{H}^{(n)}) \| \right\vert\\
 &\leq& n^{\alpha-\frac{1}{2}} \|\left(\widehat{\Psi}^{(n)}_{K_{\sigma(i)}} - \widehat{\Psi}^{(n)}_{K_{\sigma(i+1)}} \right)\left(\widehat{H}^{(n)}\right) \|  \\
 &\leq& n^{\alpha-\frac{1}{2}} \|\widehat{\Psi}^{(n)}_{K_{\sigma(i)}} - \widehat{\Psi}^{(n)}_{K_{\sigma(i+1)}} \|_\infty\| \widehat{H}^{(n)} \|,
\end{eqnarray*}
Since $\widehat{\Psi}^{(n)}_{K_{\sigma(i)}}$ and $\widehat{\Psi}^{(n)}_{K_{\sigma(i+1)}}$ converge almost surely, as $n\rightarrow +\infty$, to $\Psi_{K_{\sigma(i)}} $ and  $\Psi_{K_{\sigma(i+1)}}$ respectively, and since $ \widehat{H}^{(n)}$ converges in distribution, as $n\rightarrow +\infty$, to $H$, it follows from the preceding inequality and from $\alpha<1/2$ that  $n^\alpha\left(\widehat{\xi}^{(n)}_{K_{\sigma (i)}}-\widehat{\xi}^{(n)}_{K_{\sigma (i+1)}}\right)$ converges in probability to $0$ as $n\rightarrow +\infty$.
If $\gamma_\ell\neq 0$, we have
\begin{eqnarray*}
 n^{\alpha}\left( \widehat{\xi}^{(n)}_{K_{\sigma(i)}} -\widehat{\xi}^{(n)}_{K_{\sigma(i+1)}}\right)& = & n^{\alpha-\frac{1}{2}}\left(\|\widehat{\Psi}^{(n)}_{K_{\sigma(i)}}(\widehat{H}^{(n)}) + \sqrt{n}\delta_{K_{\sigma(i)}} \| - \| \widehat{\Psi}^{(n)}_{K_{\sigma(i+1)}}(\widehat{H}^{(n)}) + \sqrt{n}\delta_{K_{\sigma(i+1)}} \| \right) \\
&=& \frac{n^{\alpha-\frac{1}{2}} \left( \|\widehat{\Psi}^{(n)}_{K_{\sigma(i)}}(\widehat{H}^{(n)}) \|^{2} - \| \widehat{\Psi}^{(n)}_{K_{\sigma(i+1)}}(\widehat{H}^{(n)}) \|^{2}\right)}{\|\widehat{\Psi}^{(n)}_{K_{\sigma(i)}}(\widehat{H}^{(n)}) + \sqrt{n}\delta_{K_{\sigma(i)}}  \| + \| \widehat{\Psi}^{(n)}_{K_{\sigma(i+1)}}(\widehat{H}^{(n)}) + \sqrt{n}\delta_{K_{\sigma(i+1)}} \|} \\
 &+& \frac{2n^{\alpha} \left( \left\langle \delta_{K_{\sigma(i)}},\widehat{\Psi}^{(n)}_{K_{\sigma(i)}}(\widehat{H}^{(n)}) \right\rangle - \left\langle \delta_{K_{\sigma(i+1)}},\widehat{\Psi}^{(n)}_{K_{\sigma(i+1)}}(\widehat{H}^{(n)}) \right\rangle \right)}{\|\widehat{\Psi}^{(n)}_{K_{\sigma(i)}}(\widehat{H}^{(n)}) + \sqrt{n}\delta_{K_{\sigma(i)}} \| + \| \widehat{\Psi}^{(n)}_{K_{\sigma(i+1)}}(\widehat{H}^{(n)}) + \sqrt{n}\delta_{K_{\sigma(i+1)}} \|}  \\
&=& \frac{n^{\alpha-1} \left( \|\widehat{\Psi}^{(n)}_{K_{\sigma(i)}}(\widehat{H}^{(n)}) \|^{2} - \| \widehat{\Psi}^{(n)}_{K_{\sigma(i+1)}}(\widehat{H}^{(n)}) \|^{2}\right)}{\| n^{-\frac{1}{2}}\widehat{\Psi}^{(n)}_{K_{\sigma(i)}}(\widehat{H}^{(n)}) + \delta_{K_{\sigma(i)}} \| + \| n^{-\frac{1}{2}}\widehat{\Psi}^{(n)}_{K_{\sigma(i+1)}}(\widehat{H}^{(n)}) + \delta_{K_{\sigma(i+1)}} \|} \\
   &+& \frac{2n^{\alpha-\frac{1}{2}} \left( \left\langle \delta_{K_{\sigma(i)}},\widehat{\Psi}^{(n)}_{K_{\sigma(i)}}(\widehat{H}^{(n)}) \right\rangle - \left\langle \delta_{K_{\sigma(i+1)}},\widehat{\Psi}^{(n)}_{K_{\sigma(i+1)}}(\widehat{H}^{(n)}) \right\rangle \right)}{\|n^{-\frac{1}{2}}\widehat{\Psi}^{(n)}_{K_{\sigma(i)}}(\widehat{H}^{(n)}) + \delta_{K_{\sigma(j)}} \| + \|n^{-\frac{1}{2}} \widehat{\Psi}^{(n)}_{K_{\sigma(i+1)}}(\widehat{H}^{(n)}) + \delta_{K_{\sigma(i+1)}} \|},
\end{eqnarray*}
where $<\cdot , \cdot >$ is the inner prodcut defined by $<A,B>=tr(A^\ast B)$. First,
\begin{eqnarray}\label{avdernier}
& & \left| n^{\alpha-1} \left( \|\widehat{\Psi}^{(n)}_{\sigma(j)}(\widehat{H}^{(n)}) \|^{2} - \| \widehat{\Psi}^{(n)}_{\sigma(j+1)}(\widehat{H}^{(n)}) \|^{2}\right)\right|\nonumber\\
 &\leq& n^{\alpha-1}  \left( \|\widehat{\Psi}^{(n)}_{K_{\sigma(i)}}(\widehat{H}^{(n)}) \|^{2} + \|\widehat{\Psi}^{(n)}_{K_{\sigma(i+1)}}(\widehat{H}^{(n)}) \|^{2}\right)  \nonumber\\
 &\leq& n^{\alpha-1} \left( \|\widehat{\Psi}^{(n)}_{K_{\sigma(i)}} \|_\infty^{2} + \| \widehat{\Psi}^{(n)}_{K_{\sigma(i+1)}} \|_\infty^{2}\right)\|\widehat{H}^{(n)}\|^2 
\end{eqnarray}
and, further,
\begin{eqnarray}\label{dernier}
& &\left|2n^{\alpha-\frac{1}{2}}\left( \left\langle \delta_{K_{\sigma(i)}},\widehat{\Psi}^{(n)}_{K_{\sigma(i)}}(\widehat{H}^{(n)}) \right\rangle - \left\langle \delta_{K_{\sigma(i+1)}},\widehat{\Psi}^{(n)}_{K_{\sigma(i+1)}}(\widehat{H}^{(n)}) \right\rangle \right)\right|\nonumber \\
 &\leq&  2n^{\alpha-\frac{1}{2}}\left( \left|\left\langle \delta_{K_{\sigma(i)}},\widehat{\Psi}^{(n)}_{K_{\sigma(i)}}(\widehat{H}^{(n)}) \right\rangle\right| + \left|\left\langle \delta_{K_{\sigma(i+1)}},\widehat{\Psi}^{(n)}_{K_{\sigma(i+1)}}(\widehat{H}^{(n)}) \right\rangle\right| \right) \nonumber\\
 &\leq& 2n^{\alpha-\frac{1}{2}} \left( \| \delta_{K_{\sigma(i)}}\|\|\widehat{\Psi}^{(n)}_{K_{\sigma (i))}}(\widehat{H}^{(n)})\| + \| \delta_{K_{\sigma(i+1)}}\|\|\widehat{\Psi}^{(n)}_{K_{\sigma(i+1)}}(\widehat{H}^{(n)})\| \right) \nonumber\\
 &\leq& 2n^{\alpha-\frac{1}{2}}\gamma_{\ell} \left(\|\widehat{\Psi}^{(n)}\Vert_\infty +\|\widehat{\Psi}^{(n)}_{K_{\sigma(i+1)}}\|_\infty \right)\|\widehat{H}^{(n)}\|.
\end{eqnarray}
Equations (\ref{avdernier}) and (\ref{dernier}), and the above recalled  convergence properties permit to conclude that the sequence  $n^\alpha\left(\widehat{\xi}^{(n)}_{K_{\sigma (i)}}-\widehat{\xi}^{(n)}_{K_{\sigma (i+1)}}\right)$ converges in probability to $0$,  as $n\rightarrow +\infty$.
\end{proof}

}

 {\color{myaqua}

}}


\begin{thebibliography}{10}

{\color{black}

\bibitem{An13}
B. An, J. Guo, H. Wang.
\newblock Multivariate regression shrinkage and selection by canonical correlation analysis.
\newblock {\em Comput. Statist. Data Anal.}, 62:93--107, 2013.

\bibitem{Brei92}
L.~Breiman, P. Spector.
\newblock Submodel selection and evaluation in regression. The $X$-random case.
\newblock {\em Internat. Statist. Rev.}, 60:291--319, 1992.

\bibitem{Daux94}
J.~Dauxois, Y. Romain, S. Viguier.
\newblock Tensor products and statistics.
\newblock {\em Linear Algebra Appl.}, 210:59--88, 1994.

\bibitem{Fuji97}
Y. Fujikoshi, K. Sato.
\newblock Modified AIC and $C_p$ in multivariate linear regression.
\newblock {\em Biometrika}, 84:707--716, 1997.

\bibitem{Fuji11}
Y. Fujikoshi, T. Kan, S. Takahashi, T. Sakurai.
\newblock Prediction error criterion for selecting variables in a linear regression model.
\newblock {\em Ann. Inst. Stat. Math.}, 63:387--403, 2011.

\bibitem{Hock76}
R.~R. Hocking.
\newblock The analysis and selection in linear regression.
\newblock {\em Biometrics}, 32:1--49, 1976.

\bibitem{Lin86}
H.~Linhart, W. Zucchini.
\newblock {\em Model selection}.
\newblock Wiley, New York, 1986.

\bibitem{Mill90}
A.~J. Miller.
\newblock {\em Subset selection in regression}.
\newblock Chapman and Hall, London, 1990.

\bibitem{Nk01}
G.~M. Nkiet.
\newblock S\'election des variables dans un mod\`ele structurel de r\'egression lin\'eaire.
\newblock {\em C. R. Acad. sci. paris I}, 333:1105--1110, 2001.

\bibitem{Nk12}
G.~M. Nkiet.
\newblock Direct variable selection for discrimination among several groups.
\newblock {\em J.  Multivariate Anal.}, 105:151--163, 2012.

\bibitem{Shao93}
J. Shao.
\newblock Linear model selection by cross-validation.
\newblock {\em J. Amer. Statist. Assoc.}, 88:486--494, 1993.

\bibitem{Shib84}
R. Shibata.
\newblock Approximate efficiency of a selection procedure for the number of regression variables.
\newblock {\em Biometrika}, 71:43--49, 1984.

\bibitem{Thom78a}
M.~L. Thomson.
\newblock Selection of variables in multiple regression. Part I. A review and evaluation.
\newblock {\em Internat. Statist. Rev.}, 46:1--19, 1978.


\bibitem{Thom78b}
M.~L. Thomson.
\newblock Selection of variables in multiple regression. Part II. Chosen procedures, computations and examples.
\newblock {\em Internat. Statist. Rev.}, 46:129--145, 1978.

\bibitem{Zhang92}
P. Zhang.
\newblock  On the distributional properties of model selection criteria.
\newblock {\em J. Amer. Statist. Assoc.}, 87:732--737, 1992.

\bibitem{Zhang93}
P. Zhang.
\newblock  Model selection via multifold cross validation.
\newblock {\em Ann. Statist.}, 21:299--313, 1992.

\bibitem{Zheng97}
X. Zheng, W.~Y. Loh. 
\newblock  A consistent variable selection criterion for linear models with high-dimensional covariates.
\newblock {\em Statistica Sinica}, 7:311--325, 1997.


}

\end{thebibliography}
\end{document}